\documentclass[a4paper,reqno,12pt]{amsart}
\usepackage{verbatim}
\usepackage{amssymb,amsmath,amsthm}
\usepackage{amsfonts}
\usepackage{array}

\allowdisplaybreaks

\newtheorem{theorem}{Theorem}[section]
\newtheorem{lemma}{Lemma}[section]

\begin{document}

\title{\small New Multiple Harmonic Sum Identities}


\author{Helmut~Prodinger and Roberto~Tauraso}

\address{\vspace{1cm}
\begin{flushleft} 
Helmut~Prodinger \\
Department of Mathematics \\
University of Stellenbosch \\7602 Stellenbosch, South Africa\\
\vspace{0.5cm}
Roberto~Tauraso \\
Dipartimento di Matematica \\
Universit\`a di Roma ``Tor Vergata'' \\ 00133 Roma, Italy
\vspace{0.5cm}
\end{flushleft}}

\email{hproding@sun.ac.za, tauraso@mat.uniroma2.it}


\date{}


\begin{abstract} We consider a special class of binomial sums involving harmonic numbers and we prove three identities by using the elementary method of the partial fraction decomposition. Some applications to infinite series  and congruences are given.
\end{abstract}

\maketitle

\pagestyle{plain}

\section{Introduction} 

The {\sl binomial transform} of a sequence, $\{a_n\}_{n\geq 0}$, is the sequence defined by
$$n\mapsto\sum_{k=0}^n (-1)^{k}{n\choose k} a_k.$$
This kind of transform has been widely studied, see for example the pioneering paper~\cite{FlSe95}:
In it, the following integral representation (N\"orlund-Rice) is used:
\begin{equation*}
\sum_{k=0}^n (-1)^{k}{n\choose k} f(k)=\frac{(-1)^n}{2\pi i}\int_{\mathcal{C}}
\frac{n!}{z(z-1)\dots(z-n)}f(z)\,dz,
\end{equation*}
where $f(z)$ is an analytic extension of the sequence $f(k)$, and the curve $\mathcal{C}$ 
includes the poles $0,1,\dots,n$ and no others. Enlarging the contour of integration
leads to an asymptotic expansion, but often also to identities, in particular when $f(z)$
is a rational function. In this case the method is equivalent to considering the partial fraction expansion of $\frac{n!}{z(z-1)\dots(z-n)}f(z)$. This is also the point of view that we adopt in the sequel. Our terms $f(k)$ are not exactly rational functions in $k$, but not too far away.
We consider  for example {\it harmonic numbers} $H_n(s)$ with $s\in {\mathbb N}^+$, which can be written as
\begin{equation*}
H_n(s)=\sum_{j=1}^n\frac{1}{j^s}=\sum_{j=1}^\infty\Big(\frac{1}{j^s}-\frac{1}{(n+j)^s}\Big),
\end{equation*}
and for fixed $j$, this \emph{is} a rational function, leading to an identity, 
and these identities will be summed over all $j$.

In this paper we are interested in a variation of the binomial transform, that is 
$$n\mapsto\sum_{k=1}^n (-1)^{k}{n\choose k}{n+k\choose k}^{\pm 1} a_k,$$
in the special case where the generic term $a_n$ involves harmonic numbers.
As we will see in a moment, these type of sums
can be evaluated in terms of {\it multiple harmonic sum} which are defined by
\begin{equation*}
H_n(s_1, s_2, \ldots, s_l)=\sum_{1\le k_1<k_2<\cdots<k_l\le n}\prod_{i=1}^l
\frac{({\rm sgn}(s_i))^{k_i}}{k_i^{|s_i|}}.
\end{equation*}
for $l\in {\mathbb N}^+$,  ${\bf s}=(s_1, s_2, \ldots, s_l)\in 
({\mathbb Z}^{*})^l$.
The integers  $l({\bf s}):=l$ and $|{\bf s}|:=\sum_{i=1}^l|s_i|$ 
are called the length (or depth) and the weight
of a multiple harmonic sum.
By $\{s_1, s_2, \ldots, s_j\}^m$ we denote the set formed
by repeating $m$ times $(s_1, s_2, \ldots, s_j)$. 

\noindent The following three identities appear in 
\cite[Th. 1 and Th 2]{Pr:10}: 
for any positive integers $n,r$,
\begin{align} 
&\sum_{k=1}^n (-1)^{k}{n\choose k}{n+k\choose k}H_k(1)=2(-1)^n H_n(1),
\label{i5}\\
&\sum_{k=1}^n (-1)^{k}{n\choose k}{n+k\choose k}H_k(d)=
(-1)^{n-1}\sum_{{\bf s}\preceq (\{1\}^{d-2},2)} 2^{l({\bf s})}
H_n(s_1,\dots,s_{l({\bf s})-1},-s_{l({\bf s})}),
\label{i6}\\
&\sum_{k=1}^n {(-1)^{k}\over k^r}{n\choose k}{n+k\choose k}=
-\sum_{{\bf s}\preceq (\{1\}^r)} 2^{l({\bf s})}H_n({\bf s}),
\label{i7}
\end{align}
where ${\bf s}\preceq {\bf t}=(t_1, t_2, \ldots, t_m)\in ({\mathbb N}^{+})^m$ means  that the sums are taken over  
all the {\sl compositions} of {\bf t}, i. e. any ${\bf s}=(s_1, s_2, \ldots, s_l)$ such that 
$s_i=\sum_{j_{i-1}<k\leq j_{i}} t_k$ for some
$0=j_0<j_1< j_2<\cdots <j_l=m$. So, for example ${\bf s}\preceq (1,1,1,2)$ if and only if
$${\bf s}\in\{(1,1,1,2),(2,1,2),(1,2,2),(1,1,3),(2,3),(1,4),(3,2),(5)\}.$$

\noindent Moreover, let
$$S_{m,n}(s_1,\dots,s_l)=
\sum_{m\leq k_1\leq\cdots\leq k_l\leq n}
\frac{1}{k_1^{s_1}\cdots k_r^{s_l}},$$
for ${\bf s}=(s_1, s_2, \ldots, s_l)\in ({\mathbb N}^{*})^l$,
with the convention that $S_n({\bf s}))=S_{1,n}({\bf s}))$.
The following four identities appear in 
\cite[Cor. 2.1 and Th. 2.3]{HHT:13}: for any positive integers $n,a,b$,
\begin{align}
&2\sum_{k=1}^n
{(-1)^{k}\over k^{2b}} \binom{n}{k}\binom{n+k}{k}^{-1}=-S_n(\{2^b\}),
\label{i1}\\
&2\sum_{k=1}^n
{1\over k^{2b-1}} \binom{n}{k}\binom{n+k}{k}^{-1}=S_n(1,\{2^{b-1}\}),
\label{i2}\\
&2\sum_{k=1}^n
\frac{(-1)^{k}}{k^{2b}} \binom{n}{k}\binom{n+k}{k}^{-1}(H_k(2a+1)+H_{k-1}(2a+1))=-S_n(\{2^a\},3,\{2^{b-1}\}),
\label{i3}\\
&2\sum_{k=1}^n
\frac{1}{k^{2b-1}} \binom{n}{k}\binom{n+k}{k}^{-1}(H_k(-2a)+H_{k-1}(-2a))
=-S_n(\{2^a\},1,\{2^{b-1}\}).
\label{i4}
\end{align}
Note that also in this case the right-hand side is a combination of multiple harmonic sums because
$$S_n({\bf t})=\sum_{{\bf s}\preceq{\bf t}} H_n({\bf s}).$$
In the next two sections we provide three new identities which extend the pattern. 
The proofs are based on the elementary method of the partial fraction decomposition. 
In the final section we give some applications to infinite series and to congruences. In particular we show the following 
{\sl exotic sums}:
\begin{align*}
R_{\infty}(1,2)&=\sum_{1\leq j_1<k_1\leq k_2}
\frac{1}{j_1^2(2k_1-1)k_2^2}=
\frac{31}{8}\zeta(5)-\frac{7}{4}\zeta(3)\zeta(2),\\
R_{\infty}(2,2)&=\sum_{1\leq j_1\leq j_2<k_1\leq k_2}
\frac{1}{j_1^2 j_2^2(2k_1-1)k_2^2}=
-\frac{635}{64}\zeta(7)+\frac{49}{16}\zeta(3)\zeta(4)+\frac{31}{8}\zeta(5)\zeta(2).
\end{align*}
where
$$R_n(a,b)=\sum_{k=1}^n
\frac{S_{1,k-1}(\{2\}^{a})S_{k,n}(\{2\}^{b-1})}{(2k-1)}=
\sum_{1\leq j_1\leq\cdots\leq j_a
<k_1\leq\cdots\leq k_b\leq n}
\frac{1}{j_1^2\cdots j_a^2(2k_1-1)k_2^2
\cdots k_b^2}$$
for $a,b\in {\mathbb N}^+$.

\section{Results: first part}

\begin{theorem} For $n,r\geq 1$ and $d\geq 2$
\begin{align}
&\sum_{k=1}^n {(-1)^{k}\over k^r}{n\choose k}{n+k\choose k}H_k(1)=
\sum_{{\bf s}\preceq (2,\{1\}^{r-1})} 2^{l({\bf s})}
H_n(-s_1,s_2,\dots,s_{l({\bf s})}),\label{c1}\\
&\sum_{k=1}^n {(-1)^{k}\over k^r}{n\choose k}{n+k\choose k}H_k(d)=
-\sum_{{\bf s}\preceq (\{1\}^{d-2},3,\{1\}^{r-1})} 
2^{l({\bf s})}H_n({\bf s}).
\label{c2}
\end{align}
\end{theorem}
\begin{proof} For $n,j\geq 1$, let (factorials are defined via Gamma functions)
\begin{equation*}
f(z):=\frac{(z+n)!(z-n-1)!}{z!(z-1)!}\cdot
\frac{1}{z^{r+1}}\cdot\bigg[\frac1{j^d}-\frac1{(z+j)^d}\bigg],
\end{equation*}
and consider its partial fraction decomposition
\begin{align*}
f(z)&=\sum_{k=1}^n(-1)^{n-k}\binom{n}{k}\binom{n+k}{k}
\frac{1}{k^r}\bigg[\frac1{j^d}-\frac1{(k+j)^d}\bigg]\frac{1}{z-k}\\
&\qquad +\frac{C_r(n,j)}{z^r}+\cdots+\frac{C_1(n,j)}{z}
+\frac{D_d(n,j)}{(z+j)^d}+\cdots+\frac{D_1(n,j)}{z+j}.
\end{align*}
Multiply by $z$ and let $z\to\infty$. 
\begin{align*}
0&=\sum_{k=1}^n(-1)^{n-k}\binom{n}{k}\binom{n+k}{k}\frac{1}{k^r}\bigg[\frac1{j^d}-\frac1{(k+j)^d}\bigg]+C_1(n,j)+D_1(n,j)
\end{align*}
where
\begin{align*}
C_1(n,j)&=[z^{-1}]\frac{(z+n)!(z-n-1)!}{z!(z-1)!}
\frac{1}{z^{r+1}}\bigg[\frac1{j^d}-\frac1{(z+j)^d}\bigg]\\
&=[z^{r}]\frac{(z+n)!(z-n-1)!}{z!(z-1)!}
\bigg[\frac1{j^d}-\frac1{(z+j)^d}\bigg]\\
&=-[z^{r}][w^{d-1}]\frac{(z+n)!(z-n-1)!}{z!(z-1)!}
\bigg[\frac{1}{w-j}-\frac{1}{w-z-j}\bigg],
\end{align*}
\begin{align*}
D_1(n,j)&=[(z+j)^{-1}]\frac{(z+n)!(z-n-1)!}{z!(z-1)!}\frac{1}{z^{r+1}}\bigg[\frac1{j^d}-\frac1{(z+j)^d}\bigg]\\
&=-[(z+j)^{d-1}]\frac{(z+n)!(z-n-1)!}{z!(z-1)!}\frac{1}{z^{r+1}}\\
&=-[w^{d-1}]\frac{(w-j+n)!(w-j-n-1)!}{(w-j)!(w-j-1)!}
\frac{1}{(w-j)^{r+1}}\\
&=[z^{r}][w^{d-1}]\frac{(w-j+n)!(w-j-n-1)!}{(w-j)!(w-j-1)!}
\frac{1}{z-(w-j)}\\
&=[z^r][w^{d-1}]\frac{(w-j+n)!(w-j-n-1)!}{(w-j)!(w-j-1)!}
\bigg[\frac{1}{w-j}-\frac{1}{w-z-j}\bigg].
\end{align*}
Define $S(n)$ via
$$[z^{r}][w^{d-1}]S(n)=\sum_{j\geq 1}(C_1(n,j)+D_1(n,j))$$
then by summing over $j\geq 1$, we obtain 
\begin{align*}
[z^{r}][w^{d-1}](-1)^{n-1}S(n)&=
\sum_{k=1}^n\frac{(-1)^{k}}{k^r}\binom{n}{k}\binom{n+k}{k}
H_k(d).
\end{align*}
Now set
\begin{align*}
F(n,j)&=\left(-\frac{(z+n)!(z-n-1)!}{z!(z-1)!}
+\frac{(w-j+n)!(w-j-n-1)!}{(w-j)!(w-j-1)!}
\right)
\bigg[\frac{1}{w-j}-\frac{1}{w-z-j}\bigg],
\end{align*}
and
$$G(n,j)=-\frac{(w-j+n+1)!(w-j-n-1)!}{(w-j)!(w-j)!}\frac{2z}{n+1}.$$
Then, as is easy to check,
\begin{equation*}
(n+1-z)F(n+1,j)+(n+1+z)F(n,j)=G(n,j+1)-G(n,j).
\end{equation*}
Summing over $j\geq 1$, we get
$$(n+1-z)S(n+1)+(n+1+z)S(n)=\lim_{j\to\infty}G(n,j+1)-G(n,1)
=-\frac{2z}{n+1}-G(n,1).$$
Let 
$$T(n)=(-1)^n\frac{(n-z)!}{(n+z)!}S(n)$$
then, since 
$$G(n,1)=-\frac{(w+n)!(w-n-2)!}{(w-1)!(w-1)!}\frac{2z}{n+1},$$
we have
\begin{align*}
T(n+1)-T(n)&=(-1)^n\frac{(n-z)!}{(n+z)!}\frac{1}{(n+1+z)}
\left(\frac{2z}{n+1}+G(n,1)\right)\\
&=\frac{(-1)^n(n-z)!}{(n+1+z)!}
\left(1-\frac{(w+n)!(w-n-2)!}{(w-1)!(w-1)!}\right)
\frac{2z}{n+1}.
\end{align*}
Therefore
\begin{align*}
T(n)&=\sum_{k=1}^n (T(k)-T(k-1))\\
&=\sum_{k=1}^n\frac{(-1)^{k-1}(k-1-z)!}{(k+z)!}
\left(1-\frac{(w+k-1)!(w-k-1)!}{(w-1)!(w-1)!}\right)\frac{2z}{k},
\end{align*}
and
\begin{align*}
(-1)^{n-1}S(n)&=\frac{(n+z)!}{(n-z)!}
\sum_{k=1}^n\frac{(-1)^{k}(k-1-z)!}{(k+z)!}
\left(1-\frac{(w+k-1)!(w-k-1)!}{(w-1)!(w-1)!}\right)\frac{2z}{k}
\\
&=2\sum_{k=1}^n
\frac{z}{(1-z/k)}
\prod_{j=k+1}^{n}\frac{1+z/j}{1-z/j}\cdot
\left(\frac{(-1)^k}{k^2}-\frac{1}{k^3}\frac{w}{(1-w/k)}
\prod_{j=1}^{k-1}\frac{1+w/j}{1-w/j}\right).
\end{align*}
If $d=1$ then
$$\sum_{k=1}^n\frac{(-1)^{k}}{k^r}\binom{n}{k}\binom{n+k}{k}
H_k(1)$$
is given by
\begin{align*}
[z^{r}][w^{0}](-1)^{n-1}S(n)
&=2[z^{r}]\sum_{k=1}^n
\frac{z}{(1-z/k)}
\prod_{j=k+1}^{n}\frac{1+z/j}{1-z/j}\cdot
\left(\frac{(-1)^k}{k^2}\right)\\
&=2[z^{r-1}]\sum_{k=1}^n
\frac{(-1)^k}{k^2}
\bigg(1+\sum_{s=1}^{\infty}
\frac{z^s}{k^s}\bigg)
\prod_{j=k+1}^{n}\bigg(1+2\sum_{s=1}^{\infty}
\frac{z^s}{j^s}\bigg)\\
&=\sum_{{\bf s}\preceq (2,\{1\}^{r-1})} 2^{l({\bf s})}
H_n(-s_1,s_2,\dots,s_{l({\bf s})}).
\end{align*}
If $d\geq 2$ then
$$\sum_{k=1}^n\frac{(-1)^{k}}{k^r}\binom{n}{k}\binom{n+k}{k}
H_k(d)$$
is given by
\begin{align*}
[z^{r}][w^{d-1}](-1)^{n-1}S(n)
&=-2[z^{r}][w^{d-1}]\sum_{k=1}^n
\frac{z}{(1-z/k)}
\prod_{j=k+1}^{n}\frac{1+z/j}{1-z/j}\\
&\qquad\qquad\cdot
\left(\frac{1}{k^3}\frac{w}{(1-w/k)}
\prod_{j=1}^{k-1}\frac{1+w/j}{1-w/j}\right)\\
&=-2[z^{r-1}][w^{d-2}]\sum_{k=1}^n
\bigg(1+\sum_{s=1}^{\infty}
\frac{w^s}{k^s}\bigg)
\prod_{j=1}^{k-1}\bigg(1+2\sum_{s=1}^{\infty}
\frac{w^s}{j^s}\bigg)\\
&\qquad\qquad\cdot\left(\frac{1}{k^3}\right)\cdot
\bigg(1+\sum_{s=1}^{\infty}
\frac{z^s}{k^s}\bigg)
\prod_{j=k+1}^{n}\bigg(1+2\sum_{s=1}^{\infty}
\frac{z^s}{j^s}\bigg)\\
&=-\sum_{{\bf s}\preceq (\{1\}^{d-2},3,\{1\}^{r-1})} 
2^{l({\bf s})}H_n({\bf s}).
\end{align*}
\end{proof}

\section{Results: second  part}

\begin{theorem} For $n,a,b\geq 1$ 
\begin{align}\label{c3}
2\sum_{k=1}^n
\frac{1}{k^{2b-1}}\binom{n}{k}\binom{n+k}{k}^{-1}\Big((-1)^k H_{k-1}(2a)-H_{k-1}(-2a)\Big)
&=R_n(a,b).
\end{align}
\end{theorem}
\begin{proof}

\noindent For $n,j\geq 1$, let
\begin{equation*}
f(z):=\frac{z!(z-n-1)!}{(z+n)!(z-1)!}\cdot \frac{1}{z^{r+1}}\cdot \bigg[\frac1{j^d}-\frac1{(z+j)^d}\bigg].
\end{equation*}
Then its partial fraction decomposition is
\begin{align*}
f(z)&=\sum_{k=1}^n\frac{(-1)^{n-k}}{(n+k)!(n-k)!}\frac{1}{k^r}
\bigg[\frac1{j^d}-\frac1{(k+j)^d}\bigg]\frac{1}{z-k}\\
&+(-1)^r\sum_{k=1}^n\frac{(-1)^{n-k}}{(n+k)!(n-k)!}\frac{1}{k^r}\bigg[\frac1{j^d}\bigg]\frac{1}{z+k}\\
&+(-1)^r\sum_{k=1,k\not=j}^n\frac{(-1)^{n-k}}{(n+k)!(n-k)!}
\frac{1}{k^r}\bigg[-\frac{1}{(-k+j)^d}\bigg]\frac{1}{z+k}\\
&+\frac{C_r(n,j)}{z^r}+\cdots+\frac{C_1(n,j)}{z}
+\frac{D_{d+1}(n,j)}{(z+j)^{d+1}}+\cdots+\frac{D_1(n,j)}{z+j}.
\end{align*}
Multiply by $z$ and let $z\to\infty$. 
\begin{align*}
0&=\sum_{k=1}^n\frac{(-1)^{n-k}}{(n+k)!(n-k)!}\frac{1}{k^{r}}
\bigg[\frac1{j^d}-\frac1{(k+j)^d}\bigg]\\
&+(-1)^r\sum_{k=1}^n\frac{(-1)^{n-k}}{(n+k)!(n-k)!}
\frac{1}{k^{r}}\bigg[\frac{1}{j^d}\bigg]\\
&+(-1)^r\sum_{k=1,k\not=j}^n\frac{(-1)^{n-k}}{(n+k)!(n-k)!}
\frac{1}{k^{r}}\bigg[-\frac{1}{(-k+j)^d}\bigg]\\
&+C_1(n,j)+D_1(n,j).
\end{align*}
Moreover
\begin{align*}
C_1(n,j)&=-[z^{r}][w^{d-1}]\frac{z!(z-n-1)!}{(z+n)!(z-1)!}\cdot \bigg[\frac1{w-j}-\frac1{w-(z+j)}\bigg]
\end{align*}
and
\begin{align*}
D_1(n,j)&=[z^{r}][w^{d-1}]
\frac{(w-j)!(w-j-n-1)!}{(w-j+n)!(w-j-1)!}\cdot \bigg[\frac1{w-j}-\frac1{w-(z+j)}\bigg].
\end{align*}
Define $S(n)$ via
$$[z^{r}][w^{d-1}]S(n)=(n!)^2\sum_{j\geq 1}(C_1(n,j)+D_1(n,j))$$
then by summing over $j\geq 1$, we obtain 
\begin{align*}
[z^{r}][w^{d-1}](-1)^{n-1}S(n)&=
\sum_{k=1}^n\frac{(-1)^{k}}{k^{r}}\binom{n}{k}
\binom{n+k}{k}^{-1}H_k(d)\\
&+(-1)^r\sum_{k=1}^n\frac{(-1)^{k}}{k^{r}}
\binom{n}{k}\binom{n+k}{k}^{-1}\zeta(d)\\
&-(-1)^r\sum_{j=1}^n\sum_{k=1}^{j-1}\frac{(-1)^{k}}{k^{r}}
\binom{n}{k}\binom{n+k}{k}^{-1}\frac{1}{(j-k)^d}\\
&-(-1)^r\sum_{j=1}^n\sum_{k=j+1}^{n}\frac{(-1)^{k}}{k^{r}}
\binom{n}{k}\binom{n+k}{k}^{-1}\frac{1}{(j-k)^d}\\
&-(-1)^r\sum_{j=n+1}^{\infty}\sum_{k=1}^{n}\frac{(-1)^{k}}{k^{r}}\binom{n}{k}\binom{n+k}{k}^{-1}\frac{1}{(j-k)^d}\\
&=
\sum_{k=1}^n\frac{(-1)^{k}}{k^{r}}\binom{n}{k}\binom{n+k}{k}^{-1}
\Big(H_k(d)-(-1)^{r+d}H_{k-1}(d)\Big).
\end{align*}
Let
$$F(n,j)=(n!)^2\left(-\frac{z!(z-n-1)!}{(z+n)!(z-1)!}
+\frac{(w-j)!(w-j-n-1)!}{(w-j+n)!(w-j-1)!}\right)
\bigg[\frac{1}{w-j}-\frac{1}{w-(z+j)}\bigg]
$$
and
$$G(n,j)=(n!)^2\frac{(w-j-n-1)!}{(w-j+n+1)!}
\left(\frac{z+n+1}{2(n+1)}-(w-j+z+1)\right)\frac{z}{2n+1}.$$
Then $[z^{r}][w^{d-1}]F(n,j)=(n!)^2(C_1(n,j)+D_1(n,j))$,
$$\left(1-\frac{z^2}{(n+1)^2}\right) F(n+1,j)+F(n,j)=G(n,j+1)-G(n,j)$$
and by summing over $j\geq 1$, we get 
\begin{align*}\left(1-\frac{z^2}{(n+1)^2}\right) S(n+1)+S(n)&=\lim_{j\to\infty}G(n,j+1)-G(n,1)=-G(n,1).
\end{align*}
Let 
$$T(n)=(-1)^n\prod_{j=1}^n \left(1-\frac{z^2}{j^2}\right)\cdot S(n)$$
then, 
$$T(n+1)-T(n)=
\prod_{j=1}^n \left(1-\frac{z^2}{j^2}\right)\cdot(-1)^{n}G(n,1).$$
Since
\begin{align*}
G(n,1)&=(n!)^2\frac{(w-n-2)!}{(w+n)!}
\left(\frac{z+n+1}{2(n+1)}-(w+z)\right)\frac{z}{2n+1}\\
&=(-1)^n\prod_{j=1}^n \left(1-\frac{w^2}{j^2}\right)^{-1}\cdot
\left(\frac{1}{2}\left(1+\frac{z}{n+1}\right)-(w+z)\right)\frac{z}{(2n+1)w(w-n-1)},
\end{align*}
we have that
\begin{align*}
(-1)^{n-1}S(n)&=-\prod_{j=1}^n \left(1-\frac{z^2}{j^2}\right)^{-1}\cdot T(n)\\
&=-\prod_{j=1}^n \left(1-\frac{z^2}{j^2}\right)^{-1}\cdot
\sum_{k=1}^n(T(k)-T(k-1))\\
&=-\prod_{j=1}^n \left(1-\frac{z^2}{j^2}\right)^{-1}\cdot
\sum_{k=1}^n\prod_{j=1}^{k-1} \left(1-\frac{z^2}{j^2}\right)(-1)^{k-1}G(k-1,1)\\
&=
\sum_{k=1}^n\prod_{j=k}^{n} \left(1-\frac{z^2}{j^2}\right)^{-1}
\cdot\prod_{j=1}^{k-1}\left(1-\frac{w^2}{j^2}\right)^{-1}\\
&\qquad\cdot
\left(\frac{1}{2}\left(1+\frac{z}{k}\right)-(w+z)\right)
\frac{z(1+w/k)}{k(2k-1)w(1-w^2/k^2)}\\
&=\frac{z}{w}
\sum_{k=1}^n\prod_{j=k}^{n} \left(1-\frac{z^2}{j^2}\right)^{-1}\cdot
\prod_{j=1}^{k}\left(1-\frac{w^2}{j^2}\right)^{-1}\\
&\qquad\cdot
\left(\frac{1}{2(2k-1)k}-\frac{w}{2k^2}-\frac{z}{2k^2}-\frac{w^2}{(2k-1)k^2}-\frac{zw}{2k^3}\right).
\end{align*}
Hence
\begin{align}\label{gf2}
\sum_{k=1}^n\frac{(-1)^{k}}{k^{r}}
&\binom{n}{k}\binom{n+k}{k}^{-1}
(H_k(d)-(-1)^{r+d}H_{k-1}(d))\nonumber\\
&=[z^{r-1}][w^{d}]
\sum_{k=1}^n\prod_{j=k}^{n} \left(1-\frac{z^2}{j^2}\right)^{-1}\cdot
\prod_{j=1}^{k}\left(1-\frac{w^2}{j^2}\right)^{-1}\nonumber\\
&\qquad\cdot
\left(\frac{1}{2(2k-1)k}-\frac{w}{2k^2}-\frac{z}{2k^2}-\frac{w^2}{(2k-1)k^2}-\frac{zw}{2k^3}\right).
\end{align}
If $r=2b-1$  and $d=2a$ then we obtain
\begin{align*}
\sum_{k=1}^n\frac{(-1)^{k}}{k^{2b-1}}
&\binom{n}{k}\binom{n+k}{k}^{-1}
(H_k(2a)+H_{k-1}(2a))\\
&=[z^{2(b-1)}][w^{2a}]
\sum_{k=1}^n\prod_{j=k}^{n} \left(1-\frac{z^2}{j^2}\right)^{-1}\cdot
\prod_{j=1}^{k}\left(1-\frac{w^2}{j^2}\right)^{-1}\\
&\qquad\cdot
\left(-\frac{1}{2k}+\frac{1}{2k-1}\left(1-\frac{w^2}{k^2}\right)\right).
\end{align*}
Since
\begin{align*}
\sum_{k=1}^n\frac{1}{k^{2b-1}}&\binom{n}{k}\binom{n+k}{k}^{-1}
(H_k(-2a)+H_{k-1}(-2a))\\
&=[z^{2(b-1)}][w^{2a}]
\sum_{k=1}^n\prod_{j=k}^{n} \left(1-\frac{z^2}{j^2}\right)^{-1}\cdot
\prod_{j=1}^{k}\left(1-\frac{w^2}{j^2}\right)^{-1}\cdot\left(-\frac{1}{2k}\right)\\
\end{align*}
and
$$(-1)^{k}(H_k(2a)+H_{k-1}(2a))-(H_k(-2a)+H_{k-1}(-2a))=2((-1)^k H_{k-1}(2a)-H_{k-1}(-2a))$$
it follows that
\begin{align*}
2\sum_{k=1}^n
\frac{1}{k^{2b-1}}&\binom{n}{k}\binom{n+k}{k}^{-1}
((-1)^k H_{k-1}(2a)-H_{k-1}(-2a))\\
&=[z^{2(b-1)}][w^{2a}]
\sum_{k=1}^n\prod_{j=k}^{n} \left(1-\frac{z^2}{j^2}\right)^{-1}\cdot
\prod_{j=1}^{k-1}\left(1-\frac{w^2}{j^2}\right)^{-1}\cdot\left(\frac{1}{2k-1}\right)\\
&=\sum_{k=1}^n
\frac{S_{1,k-1}(\{2\}^{a})S_{k,n}(\{2\}^{b-1})}{(2k-1)}=R_n(a,b).
\end{align*}
\end{proof}

Note that if $r=2b-1$ and $d=1$ then \eqref{gf2} yields \eqref{i2},
\begin{align*}
\sum_{k=1}^n\frac{(-1)^{k}}{k^{2b-1}}
&\binom{n}{k}\binom{n+k}{k}^{-1}
\left(\frac{1}{k}\right)\\
&=[z^{2(b-1)}][w^{1}]
\sum_{k=1}^n\prod_{j=k}^{n} \left(1-\frac{z^2}{j^2}\right)^{-1}\cdot
\prod_{j=1}^{k}\left(1-\frac{w^2}{j^2}\right)^{-1}\cdot
\left(-\frac{w}{2k^2}\right)\\
&=-\frac{1}{2}S_n(\{2\}^{b}).
\end{align*}
Moreover, if $r=2b$ and $d=2a+1$ then \eqref{gf2} yields \eqref{i3},
\begin{align*}
\sum_{k=1}^n\frac{(-1)^{k}}{k^{2b}}
&\binom{n}{k}\binom{n+k}{k}^{-1}
(H_k(2a+1)+H_{k-1}(2a+1))\\
&=[z^{2b-1}][w^{2a+1}]
\sum_{k=1}^n\prod_{j=k}^{n} \left(1-\frac{z^2}{j^2}\right)^{-1}\cdot
\prod_{j=1}^{k}\left(1-\frac{w^2}{j^2}\right)^{-1}\cdot
\left(-\frac{zw}{2k^3}\right)\\
&=-\frac{1}{2}S_n(\{2\}^{a},3,\{2\}^{b-1}).
\end{align*}

\section{Two applications of identity \eqref{c3}}

Our first application involves infinite series.
Evaluations for Euler sums of length two and of odd weight $m+n$,
$H_{\infty}(m,-n)$ and $H_{\infty}(-m,n)$
in terms of zeta values are well known 
(see, for example, \cite[Th. 7.2]{FS:98}):
\begin{align}
2H_{\infty}(m,-n)&=
\overline{\zeta}(m+n)-
(1-(-1)^m)\zeta(m)\overline{\zeta}(n)\nonumber \\
&\quad +2(-1)^m
\sum_{r=1}^{(m+n-1)/2}
\binom{2r}{n-1}\overline{\zeta}(2r+1)\overline{\zeta}(m+n-2r-1)
\nonumber\\
&\quad -2(-1)^m\sum_{r=1}^{(m+n-1)/2}\binom{2r}{m-1}\zeta(2r+1)\overline{\zeta}(m+n-2r-1) \label{i1},\\
2H_{\infty}(-n,m)&=-2H_{\infty}(m,-n)-2\zeta(m)\overline{\zeta}(n)+2\overline{\zeta}(m+n) \label{i2},
\end{align}
where $\overline{\zeta}(0)=1/2$, $\overline{\zeta}(1)=\ln(2)$, 
$\overline{\zeta}(n)=(1-2^{1-n})\zeta(n)$, and $\zeta(n)=\sum_{k=1}^{\infty}1/k^n$, for $n>1$.

\begin{theorem}  
Let $a,b$ positive integers with $b>1$. Then
\begin{align*}
R_{\infty}(a,b)=4\sum_{r=1}^{a+b-1}
\left(\binom{2r}{2b-2}-\binom{2r}{2a-1}\right)
\left(1-\frac{1}{2^{2r+1}}\right)
\zeta(2r+1)\overline{\zeta}(2(a+b-1-r))
\end{align*}
\end{theorem}
\begin{proof} By taking the limit $n\to \infty$ in \eqref{c3},
$$R_{\infty}(a,b)=2H_{\infty}(2a,-(2b-1))-2H_{\infty}(-2a,2b-1).$$
Then use \eqref{i1} and \eqref{i2}.
\end{proof}

\noindent For the second application of \eqref{c3}, we first recall some congruences.
Let $m,n$ be positive integers and let $p$ be any prime $p>w+1$ where
$w=m+n$,

\noindent i) if $w$ is even then (\cite[Th. 3.2]{Zh:08})
$$
H_{p-1}(m,n)\equiv \left(
(-1)^{m}n\binom{w+1}{m}-(-1)^{m}m\binom{w+1}{n}-w\right) 
\frac{pB_{p-w-1}}{2(w+1)} \pmod{p^2},
$$

\noindent ii) if $m$ is even and $n$ is even then 
(\cite[Lemma 3.1]{HHT:13})
$$
H_{p-1}(-m,-n)\equiv\left(\frac{(m-n)(1-2^{-w})}{2(m+1)(n+1)}
\binom{w}{m}-\frac{w}{2(w+1)}\right)pB_{p-w-1} \pmod{p^2}.
$$

\noindent iii) if $w$ is odd then (\cite[Lemma 1]{HHT:11})
$$
H_{\frac{p-1}{2}}(m,n)\equiv \left((-1)^n
\binom{w}{m}+2^{w}-2\right)\frac{B_{p-w}}{2w} \pmod{p},
$$

\noindent iv) if $m$ is even and $n$ is odd then 
(\cite[Lemma 3.2]{HHT:13})
$$
H_{\frac{p-1}{2}}(-m,-n)\equiv 
\left(\frac{1}{2^{w}}\binom{w}{m}+1\right)\frac{(2^{w-1}-1)B_{p-w}}{w}\pmod{p}.
$$

\begin{lemma}  
Let $m,n,r$ be positive integers. If $n+r$ is odd and $m$ is even then for any prime $p>\max(m,n)+1$,
\begin{align*}
\sum_{k=1}^{p-1}
\frac{(H_{k-1}(m)-(-1)^kH_{k-1}(-m))(H_{k}(n)+H_{k-1}(n))}{k^{r}}\equiv 0 \pmod{p}.
\end{align*}
\end{lemma}
\begin{proof} Let $a$ be a positive integer or a negative even integer then, by Wolstenholme's theorem, for $p>a+1$,
$$H_{p-1}(a)\equiv 0 \pmod{p}.$$ 
Thus
\begin{align*}
H_{p-k}(a)&=\sum_{j=1}^{p-k}\frac{({\rm sgn}(a))^j}{j^a}=
\sum_{j=k}^{p-1}\frac{({\rm sgn}(a))^{p-j}}{(p-j)^a}\\
&\equiv (-1)^a{\rm sgn}(a)\sum_{j=k}^{p-1}\frac{({\rm sgn}(a))^{j}}{j^a}\\
&\equiv (-1)^a{\rm sgn}(a)(H_{p-1}(a)-H_{k-1}(a))\\
&\equiv -(-1)^a{\rm sgn}(a)H_{k-1}(a) \pmod{p}.
\end{align*}
Hence
\begin{align*}
L&:=\sum_{k=1}^{p-1}
\frac{(H_{k-1}(m)-(-1)^kH_{k-1}(-m))(H_{k}(n)+H_{k-1}(n))}{k^{r}}\\
&=\sum_{k=1}^{p-1}
\frac{(H_{p-k-1}(m)-(-1)^{p-k}H_{p-k-1}(-m))(H_{p-k}(n)+H_{p-k-1}(n))}{(p-k)^{r}}\\
&\equiv\sum_{k=1}^{p-1}
\frac{(-H_{k}(m)+(-1)^{k}H_{k}(-m))(-1)^n(-H_{k-1}(n)-H_{k}(n))}
{(-1)^r k^{r}}\\
&\equiv(-1)^{n+r}L+(-1)^{n+r}\sum_{k=1}^{p-1}
\frac{(1-(-1)^{2k})(H_{k-1}(n)-H_{k}(n))}{k^{r+m}}\\
&\equiv -L\pmod{p},
\end{align*}
which implies that $L\equiv 0 \pmod{p}$.
\end{proof} 
\begin{theorem}  
For any prime $p>2a+2b-1$,
\begin{align*}
R_{(p-1)/2}(a,b)
&\equiv -\frac{2(1-2^{-(2a+2b-1)})}{2a+2b-1}\binom{2a+2b-1}{2a}B_{p-(2a+2b-1)} \pmod{p},\\
R_{p-1}(a,b)
&\equiv \frac{2(a-b+1)R_{(p-1)/2}(a,b)}{(2a+1)} \pmod{p}.
\end{align*}
\end{theorem}
\begin{proof} 
Setting $n=\frac{p-1}{2}$, we note that
$$
\binom{n}{k}\binom{n+k}{k}^{-1}=(-1)^k
\frac{\bigl(\frac{1}{2}-\frac{p}{2}\bigr)\bigl(\frac{3}{2}-\frac{p}{2}\bigr)\cdots\bigl(\frac{2k-1}{2}-\frac{p}{2}\bigr)}%
{\bigl(\frac{1}{2}+\frac{p}{2}\bigr)\bigl(\frac{3}{2}+\frac{p}{2}\bigr)\cdots\bigl(\frac{2k-1}{2}+\frac{p}{2}\bigr)}\equiv (-1)^k \pmod{p}.
$$
Therefore, by i) and ii),
\begin{align*}
R_{(p-1)/2}(a,b)&\equiv 2\sum_{k=1}^n
\frac{1}{k^{2b-1}}(H_{k-1}(2a)-(-1)^kH_{k-1}(-2a))\\
&\equiv 2H_{(p-1)/2}(2a,2b-1)-2H_{(p-1)/2}(-2a,-(2b-1))\\
&\equiv -\frac{2(1-2^{-(2a+2b-1)})}{2a+2b-1}\binom{2a+2b-1}{2a}B_{p-(2a+2b-1)} \pmod{p}.
\end{align*}
Now, setting $n= p-1$ we have that
\begin{align*}
\binom{n}{k}\binom{n+k}{k}^{-1}&
=\frac{(-1)^k k}{p}\prod_{j=1}^k\left(1-\frac{p}{j}\right)\cdot\prod_{j=1}^{k-1}\left(1+\frac{p}{j}\right)^{-1} \\
&\equiv \frac{(-1)^k k}{p}(1-pH_k(1))(1-pH_{k-1}(1)) \\
&\equiv (-1)^k\left(\frac{k}{p}-(H_{k}(1)+H_{k-1}(1))\right)\pmod{p},
\end{align*}
and by the previous lemma, by iii) and iv), it follows that
\begin{align*}
R_{(p-1)/2}(a,b)&\equiv 
2\sum_{k=1}^n
\frac{1}{k^{2b-1}}\left(\frac{k}{p}-(H_{k}(1)+H_{k-1}(1))\right)
(H_{k-1}(2a)-(-1)^kH_{k-1}(-2a))\\
&\equiv 
\frac{2}{p}\sum_{k=1}^n
\frac{H_{k-1}(2a)-(-1)^kH_{k-1}(-2a)}{k^{2b-2}}\\
&\equiv 2\left(\frac{H_{p-1}(2a,2b-2)}{p}-\frac{H_{p-1}(-2a,-(2b-2))}{p}\right)\\
&\equiv \frac{2(a-b+1)R_{(p-1)/2}(a,b)}{(2a+1)} \pmod{p}.
\end{align*}
\end{proof}

\end{document}